\documentclass[12pt]{article}
\usepackage[T1]{fontenc}
\usepackage[utf8]{inputenc}
\usepackage[english]{babel}
\usepackage{authblk}

\usepackage{hyperref}
\usepackage{lmodern}
\usepackage{amsthm}
\usepackage{pdfpages}
\usepackage{listings}
\usepackage{multirow}
\usepackage{color, colortbl}
\usepackage{colortbl}
\usepackage{amsmath}
\usepackage{setspace}
\usepackage[noend]{algpseudocode}
\usepackage{enumerate}
\usepackage{mathtools,amssymb,amsthm}
\usepackage{color}
\usepackage{enumitem}
 \usepackage{mathtools}
 \usepackage{epigraph}
\usepackage{fancyhdr}
\usepackage{amsfonts}
\usepackage{mathrsfs}
\usepackage{tikz-cd}
\usepackage{fancybox}
\usepackage[all]{xy}
\usepackage{pifont}
\newtheorem{thm}{Theorem}[section]
\newtheorem{cor}{Corollary}[section]

\newtheorem{ex}{Example}[section]
\newtheorem{remark}{Remark}[section]

\newtheorem{lemma}{Lemma}[section]
\newtheorem{defo}{Definition}[section]

\newtheorem{prop}{Proposition}[section]

\usepackage{blindtext}
\title{\textbf{ON GRADED NIL-GOOD RINGS}}
\author{ I. Namrok\thanks{ismailnamrok1993@gmail.com}~,~ H. Choulli\thanks{hananchoulli@hotmail.com}~ and H. Mouanis\thanks{hmouanis@yahoo.fr} }
\affil{Faculty of Sciences Dhar El Mahraz, Sidi Mohamed Ben Abdellah
 University, Fez, Morocco.}
\date{}

\begin{document}
\maketitle
\begin{abstract} In this paper we introduce and study the notion of a graded nil-good ring which is graded by a group. We investigate extensions of graded nil-good rings to graded group rings. Further, we discuss graded matrix ring extensions and trivial extensions of graded nil-good rings. Furthermore, we show that the class of graded rings which are nil-good and the class of graded nil-good rings are not comparable. Moreover, we discuss the question whether or not the nil-good property of the component, which corresponds to the identity element of the grading group, implies that the whole graded ring is a graded nil-good ring.   \end{abstract}
\vskip0.1in
\textbf{keywords}: Graded rings and modules, nil-good rings, group rings, matrix rings, trivial ring extensions.
\section{Introduction}

In 1977, W.K Nicholson has introduced in \cite{2} a new class of rings called \textit{clean rings} in which every element
can be written as a sum of an idempotent and a unit. Since then, many works have been done about rings in which elements can be written as a sum of two elements with certain properties. In particular, some authors have investigated rings in which elements can be written as a sum of a nilpotent element and an element with a certain property. As an example of these rings there is, \textit{nil-clean rings}, \textit{fine rings} and \textit{nil-good rings} introduced respectively in \cite{12}, \cite{22} and \cite{1}. 

Some authors have given  graded versions of some of the previous classes of rings, such as \textit{Graded
nil-clean rings}  introduced in \cite{15}, and \textit{graded 2-nil-good rings}  introduced in \cite{16}.\\
 In this work, we define  and study \textit{graded nil-good rings} as a graded version of  nil-good rings introduced in \cite{1}. A \textit{nil-good} ring is defined as a ring whose every element is either nilpotent or a sum of a unit and a nilpotent. This class of rings is a
generalization of the notion of
fine rings (see \cite{22}), whose every nonzero element
can be written as a sum of a unit and a nilpotent element.\\
Our objective in this paper is to introduce and study the graded version of nil good ring. Thus, we define 
 \textit{graded nil-good ring} as a group graded ring in which every homogeneous element is either nilpotent or can be written as a sum of a homogeneous unit and a homogeneous nilpotent. We first give some properties of graded nil-good rings which
represent graded versions of results on nil-good rings. Moreover, we discuss when the graded group ring is graded nil-good. This leads to an interesting question of how the graded
nil-good property of a group graded ring depends on the nil-good property of the
component which corresponds to the identity element of the grading group. It is proved
that the nil-good property of the component corresponding to the identity element of
the grading group does not imply the graded nil-good property of the whole graded
ring in general. However, under some extra hypothesis, this implication becomes true. Also, many results concerning trivial extensions of such rings have been obtained. Finally, we give  a sufficient condition for the graded matrix ring over a graded commutative nil-good ring to be graded nil-good.
\section{Preliminaries }
~~~~All rings are assumed to be associative with identity. If $R$ is a ring then $J(R)$ denotes the Jacobson radical of $R$, $U(R)$ is the multiplicative group of units of $R$, and $Nil(R)$ denotes the set of nilpotent elements of $R$.

Let $R$ be a ring, $G$ a group with identity element $e$, and let $\{R_g\}_{g\in G}$ be a family of additive subgroups of $R$. $R$ is said to be $G$-graded if $R= \bigoplus_{g\in G}R_g$ and $R_gR_h\subseteq R_{gh}$ for all $g,h\in G$. The set $H= \bigcup_{g\in G}R_g $ is called the homogeneous part of $R$, elements of $H$ are called homogeneous, and subgroups $R_g$ ($g\in G$) are called components. If $a\in R_g$, then we say that $a$ has the degree $g$.

A right ideal (left, two-sided) $I$ of a $G$-graded ring $R= \bigoplus_{g\in G}R_g$ is called homogeneous or graded if $I= \bigoplus_{g\in G}I\cap R_g$. If $I$ is a two-sided homogeneous ideal, then $R/I$ is a $G$-graded ring with components $(R/I)_g=R_g/I\cap R_g$. A graded ring $R$ is graded-nil if every homogeneous element of $R$ is nilpotent, and a homogeneous ideal $I$ is called \textit{graded-nil} if every homogeneous element of $I$ is nilpotent.

If $R= \bigoplus_{g\in G}R_g$ is a $G$-graded ring, then a  $G$-graded $R$-module is an $R$-module $M$ such that $M= \bigoplus_{g\in G}M_g$, where $M_g$ are additive subgroups of $M$, and such that $R_hM_g\subseteq M_{hg}$ for all $g,h\in G$. A submodule $N$ of a $G$-graded $R$-module $M= \bigoplus_{g\in G}M_g$ is called homogeneous if $N= \bigoplus_{g\in G}N\cap M_g$.

A homogeneous right ideal $\mathfrak{m}$ of a graded ring $R$ is said to be graded-maximal right ideal if it is contained in no other proper homogeneous right ideal of $R$. A ring $R$ is graded-local if it has a unique graded-maximal right ideal.  For more  details on graded rings theory, we refer to \cite{14}.

The graded Jacobson radical $J^g(R)$ of a $G$-graded ring $R$ is defined to be the intersection of all graded-maximal right ideals of $R$. Moreover, $J^g(R)$ is a homogeneous two-sided ideal (see for instance \cite[Proposition 2.9.1]{14}).

Let $R= \bigoplus_{g\in G}R_g$ be a $G$-graded ring. According to \cite{13}, we have that the group ring $R[G]$ is $G$-graded with the $g$-component $(R[G])_g= \sum_{h\in G}R_{gh^{-1}}h$ and with the multiplication defined by $(r_gg')(r_hh')=r_gr_h(h^{-1}g'hh')$, where $g,g',h, h' \in G$, $r_g \in R_g$ and $r_h\in R_h$.

Let $A$ be a commutative ring, $E$ an $A$-module and $R:=A\propto E$ the set of pairs $(a,e)$ with pairwise addition and multiplication given by $(a,e)(b,f)=(ab,af+be)$. $R$ is called the \textit{trivial ring extension} of $A$ by $E$. Considerable work has been concerned with trivial ring extensions of commutative rings. Part of it has been summarized in Glaz's book \cite{4} and Huckaba's book \cite{20}. In this paper, we consider the same construction for noncommutative rings. In \cite{17}, it has been proved that if $R=A\propto E $ where $A$ is a commutative ring, then $U(R)=U(A)\propto E$ and $Nil(R)=Nil(A)\propto E$. We can check easily that these properties hold true even if $A$ is a non commutative ring. 

Let $A$ be a $G$-graded ring and $E$ a $G$-graded $A$-module. According to \cite[Section 3]{17},  the trivial ring extension $R=A\propto E$ is $G$-graded where $R_g=A_g\oplus E_g$.

If $R$ is a $G$-graded ring and $n$ a natural number, then the matrix ring $M_n(R)$ can be seen as a $G$-graded ring in the following manner.\\ Let $\overline{\sigma}=(g_1,\dots,g_n)\in G^n$, $\lambda\in G$ and $M_n(R)_{\lambda}(\overline{\sigma})=(a_{ij})_{n\times n}$, where\\ $a_{ij}\in R_{g_i\lambda g_j^{-1}}$, $i,j \in \{1,\dots,n\}$. Then, $M_n(R)=\displaystyle \bigoplus_{\lambda \in G}M_n(R)_{\lambda}(\overline{\sigma})$ is a $G$-graded ring with respect to usual matrix  addition and multiplication. This ring is denoted by $M_n(R)(\overline{\sigma})$.\\Note that if $\overline{\sigma}=(e,e,\dots,e)\in G^n$, then $M_n(R)_{\lambda}(\overline{\sigma})=\begin{pmatrix}
R_{\lambda} &R_{\lambda}&\dots&R_{\lambda}\\R_{\lambda}&R_{\lambda}&\dots&R_{\lambda}\\ \vdots&\vdots&\dots&\vdots \\ R_{\lambda}&R_{\lambda}&\dots&R_{\lambda} \end{pmatrix}$.

\section{Graded nil-good rings}

  ~~Let $G$ be a group with identity $e$.
\begin{defo}
A homogeneous element of a $G$-graded ring is said to be \textit{graded nil-good} if it is either nilpotent or it can be written as a sum of a homogeneous unit and a homogeneous nilpotent. A $G$-graded ring is said to be graded nil-good if each of its homogeneous elements is graded nil-good.
\end{defo}
\begin{ex}\label{ex3.1}
Let $G=\{e,g\}$ be a cyclic group of order 2 and $R:= \mathbb{Z}_2[X]/(X^2)$. We have that $R= \mathbb{Z}_2\bigoplus \mathbb{Z}_2\overline{X} $ is a $G$-graded ring. Since $\mathbb{Z}_2$ is a nil-good ring (see \cite[Example 1]{1}) and every element of $\mathbb{Z}_2\overline{X}$ is nilpotent, then $R$ is a graded nil-good ring.
\end{ex}
\begin{remark}\label{r3.1}

Let $R=\bigoplus_{g\in G}R_g$ be a $G$-graded ring. If $a=u+n$ where $u$ (resp. $n$) is a homogeneous unit (resp. nilpotent), then $a$, $u$ and $n$ are all of the same degree. Indeed, assume that the degree of $a$ is $g$. If we suppose that the degree of $u$ is not $g$ we will have $a=0$ or $a=n$. Both cases lead to a contradiction. Hence,  $u$ and $g$ have the same degree $g$.
\end{remark}
\begin{prop}\label{p3.1}
Let $A$ be a nil-good ring. Then, the Laurent polynomial ring  $R:= \bigoplus_{n\in \mathbb{Z}}AX^n$ (with $R_0=A$) is a $\mathbb{Z}$-graded nil-good ring.
\end{prop}

\begin{proof}
Let $aX^n$ ($n\in \mathbb{Z}$) be a homogeneous element of $R$. Since $A$ is nil-good, then $a=u+b$ where $u\in U(A)\cup \{0\}$ and $b\in Nil(A)$, and so $aX^n=uX^n+bX^n$, since $uX^n\in U(R)\cup \{0\}$ and $bX^n\in Nil(R)$ and both are homogeneous, then $R$ is graded nil-good.

\end{proof}
\begin{ex}\label{e3.2}
According to the previous proposition, the  $\mathbb{Z}$-graded domain $R= \bigoplus_{n\in \mathbb{Z}}\mathbb{Z}_2X^n$ is graded nil-good.\\Let us notice that $R$ is not a nil-good ring since $1+X$ is not a nil-good element. Indeed, if we suppose that $1+X$ is nil-good. Then, since $R$ is a domain $Nil(R)=(0)$, and so $1+X$ must be a unit in $R$. According to \cite[page 1]{5}, we have that $1+X$ is a homogeneous which is a contradiction. Which prove that a graded nil-good ring is not necessary a nil good ring.
\end{ex}
\begin{ex}\label{e3.3}
Let $G=\{e,g\}$ be a cyclic group of order 2, and  $R:=M_2(\mathbb{Z}_2)$.\\Thus,  $R=\begin{pmatrix} 
\mathbb{Z}_2 & 0 \\
0& \mathbb{Z}_2
\end{pmatrix} \bigoplus \begin{pmatrix} 
0 & \mathbb{Z}_2 \\
\mathbb{Z}_2 & 0
\end{pmatrix} $ is a $G$-graded ring. $R$ is not a graded nil-good ring since the homogeneous element $\begin{pmatrix}
1 & 0 \\ 0 & 0
\end{pmatrix} \in R_e$ is not graded nil-good. But $R$ is a nil-good ring (see \cite[Example 2]{1}), hence a nil-good ring is not necessarily a graded nil-good ring. 
\end{ex}
According to Example \ref{e3.2} and Example \ref{e3.3}, we deduce that the class of graded rings which are nil-good and the class of graded nil-good rings are not comparable.

\begin{prop}\label{p3.2}
\begin{enumerate}
  \item Let $R=\bigoplus_{g\in G}R_g$ be a $G$-graded nil-good ring. Then, $R_e$ is a nil-good ring. 

\item Let $R$ be a $G$-graded commutative ring. If $R$ is graded nil-good, then every homogeneous element of $R$ is either unit or nilpotent.
\item Let $R=\bigoplus_{g\in G}R_g$ be a $G$-graded nil-good ring. If $U(R)=U(R_e)$, then  every $x\in R_g$ ($g\neq e$)  is nilpotent.
\end{enumerate}
\end{prop}
\begin{proof}
\begin{enumerate}
    \item Let $a\in R_e$ be a non nilpotent element. Since $R$ is graded nil-good, then $a=u+n$ where $u\in U(R)\cap R_g$ and $n\in Nil(R)\cap R_h$ for some $g,h\in G$. Suppose that $g\neq e$, in this case we will have that $a$ is nilpotent which is a contradiction, Hence, $g=e$, and so $u,n\in R_e$. Thus, $R_e$ is a nil-good ring.
    \item Since $R$ is a commutative ring, then by \cite[page 10]{21} we have that $U(R)+Nil(R)=U(R)$ . Hence, if $R$  is commutative  graded nil-good, then every homogeneous element of $R$ is either nilpotent or unit.
    \item Let $x \in R_g$ where $g\neq e$. Since $R$ is graded nil-good, then  $x=u+n$ where $u\in U(R_e)\cup \{0\}$ and $n$ is homogeneous nilpotent element of $R$. Suppose that $u\neq 0$, so by comparing degrees we have that $x=0$ or $x=n$, this implies that $u=-n$ which is a contradiction since $u\in U(R_e)$. Hence, $u=0$ and so $x=n \in Nil(R)$.   
\end{enumerate}

\end{proof}

\begin{remark}\label{r2.2}
Let $A$ be a ring and $R:=A[X]$ be the polynomial ring with the $\mathbb{Z}$-grading $R_n = AX^n$ for $n \geq 0$ and $R_n = 0$ for $n <0$. By Proposition\ref{p3.2}(3) we have that $R$ could never be a graded nil-good ring since $X$ is not nilpotent. 
\end{remark}

\begin{prop}\label{p3.3}
If $R$ is a $G$-graded nil-good ring with $U(R)=\{1\}$, then $R=R_e \cong \mathbb{Z}_2$.
\end{prop}
\begin{proof}
Since $1+Nil(R) \subseteq U(R)$, then $Nil(R)=\{0\}$. On the other hand, since $U(R)=U(R_e)$ then by Proposition \ref{p3.2} (3), each $a\in R_g$ ($g\neq e$) is nilpotent and so $a=0$  Since $Nil(R)=0$. Hence, $R_g=0$ for all $g\neq e$ and so $R=R_e$. On the other hand, $R_e$ is nil-good by proposition \ref{p3.2}. Therefore, by \cite[Proposition 2.3]{1} we have that $R_e\cong \mathbb{Z}_2$. Which completes the proof.

\end{proof}
In \cite{1} it is proved that $R$ is a nil-good ring if and only if $R/I$ is nil-good, whenever $I$ is a nil ideal of $R$. Here we prove an analogous result in the graded nil-good case:
\begin{thm}\label{t3.1}
Let $R$ be a $G$-graded ring and $I$ a graded-nil ideal of $R$. Then $R$ is a graded nil-good ring if and only of $R/I$ is graded nil-good.  
\end{thm}
\begin{proof}
If $R$ is graded nil-good, then $R/I$ is also graded nil-good as a graded homomorphic image if $R$.

Conversely, let $R/I$ be a graded nil-good ring and $x\in R_g$ where $g\in G$. We have two cases:\\
- \textbf{Case 1} : $\Bar{x}$ is nilpotent in $R/I$. Since $I$ is graded-nil, we have $x$ is also nilpotent in $R$.\\ - \textbf{Case 2} : $\Bar{x}=\Bar{u}+\Bar{n}$, where $\Bar{u}$ is homogeneous unit of $R/I$ and $n$ is nilpotent homogeneous element of $R$ (case 1). Since $I$ is graded-nil ideal, it is contained in the graded Jacobson radical $J^g(R)$. On the other hand, $\overline{x-n}$ is unit in $R/I$. Then it is also a  unit in $R/J^g(R)$ since $I\subseteq J^g(R)$. Now, by \cite[Proposition 2.9.1]{14} $x-n$ is a unit of $R$. Finally $R$ is a graded nil-good ring.

\end{proof}
The example below shows that the condition "$I$ is a graded-nil ideal" can't be removed.

\begin{ex}\label{e2.4}
Let $R:=\mathbb{Z}_2[X]$ and $I=(X^2)$. We have that $R$ is a $\mathbb{Z}$-graded ring and $I$ is not graded-nil. By Example \ref{ex3.1}, $R/I$ is graded nil-good but $R$ is not graded nil-good (by remark \ref{r2.2}).
\end{ex}

\begin{prop}\label{p3.4}
Let $R= \bigoplus_{g\in G} R_g$ be a $G$-graded nil-good ring of finite support. Then $J^g(R)$ is a graded-nil ideal.
\end{prop}
\begin{proof}
Since $R$ is graded nil-good, then $R_e$ is a nil-good ring. According to \cite[Proposition 2.5]{1} applied to the ring $R_e$, $J(R_e)$ is nil. Now, Corollary 2.9.3 of \cite{14} implies that $J(R_e)=J^g(R)\cap R_e$. Therefore, if $a\in J^g(R)\cap R_e$ then it is nilpotent. On the other hand, if $a\in J^g(R)\cap R_g$ where $g\neq e$. Since the support of $R$ is finite then by \cite[Corollary 2.9.4]{14} $J^g(R)\subseteq J(R)$, hence $a\in J(R)$. Now, suppose that $a$ is not nilpotent. Then $a=n+u$ where $n\in Nil(R)$ and $u\in U(R)$ (since $R$ is graded nil-good). Therefore, $n=-u(1-u^{-1}a)$. Hence, $n$ is a unit of $R$ which is  a contradiction since $n$ is nilpotent. Finally, $a$ is nilpotent and so $J^g(R)$ is a graded-nil ideal.
\end{proof}
\begin{cor}\label{c3.1}
Let $R$ be a $G$-graded ring of finite support. Then $R$ is graded nil-good if and only if $J^g(R)$ is graded-nil and $R/J^g(R)$ is graded nil.
\end{cor}
\begin{remark}\label{r3.3}
By Theorem \ref{t3.1}, if $J^g(R)$ is graded-nil and $R/J^g(R)$ is graded nil-good, then $R$ is graded nil-good for any cardinality of the support of $R$.
\end{remark}
\begin{lemma}\label{l3.1}
Let $R= \bigoplus_{g\in G} R_g$ be a  commutative $G$-graded nil-good ring. Then $J^g(R)$ is a graded-nil ideal.
\end{lemma}

\begin{proof}
Since $R$ is commutative graded nil-good, then by Proposition \ref{p3.2} (2), every homogeneous element of $R$ is either nilpotent or unit. Since $J^g(R)$ is a proper ideal of $R$, then every homogeneous element of $J^g(R)$ is nilpotent. Hence, $J^g(R)$ is graded-nil.

\end{proof}
\begin{cor}\label{c3.2}
Let $R$ be a $G$-graded commutative ring. Then $R$ is graded nil-good if and only if $J^g(R)$ is graded-nil and $R/J^g(R)$ is graded nil-good.   

\end{cor}

\begin{prop}\label{p3.5}
Let $R$ be a $G$-graded-local ring of finite support. Then, $R$ is graded nil-good if and only if $J^g(R)$ is a graded-nil ideal.
\end{prop}

\begin{proof}
The necessity follows directly from proposition \ref{p3.4}. To prove the sufficiency, given a homogeneous element $a$, we have that $a\in U(R)$ or $a\in J^g(R)$. In the second case, $a$ will be nilpotent since $J^g(R)$ is graded-nil. Hence, every homogeneous element is either nilpotent or unit. Thus, $R$ is graded nil-good.

\end{proof}

\begin{remark}
Let's return to Proposition \ref{p3.2}. If $R=\bigoplus_{g\in G}R_g$ is a $G$-graded nil-good ring, then $R_e$ is a nil-good ring. Now, this yields to the question of when the following implication holds true:
$$ R_e ~\text{is nil-good}\Longrightarrow R= \bigoplus_{g\in G}R_g~\text{is graded nil-good}.$$
The following example shows that the above  implication  does not hold in general.
\end{remark}

\begin{ex}\label{e4.2}
Let $R:=\mathbb{Z}_2[X]$. $R$ is $\mathbb{Z}$-graded with $R_i=\mathbb{Z}_2X^i$ if $i\geq 0 $ and $R_i=0$ if $i< 0$. Then $R_0=\mathbb{Z}_2$ is a nil-good ring (see \cite[Example 1]{1}), but $R$ is not graded nil-good since $X$ is nilpotent.
\end{ex}
Next we give some sufficient conditions for the above implication to be true. Let us first recall the definition of $PI$-ring.
\begin{defo}[\cite{23}]
A ring $R$ is a $PI$-ring if there is, for some natural integer $n$ an element $P$ of $\mathbb{Z}[X_1,\dots,X_n]$ such that for all $(r_1,\dots,r_n)\in R^n$ we have that $P(r_1,\dots,r_n)=0$.

\end{defo}

\begin{thm}\label{t4.5}
Let $R= \bigoplus_{g\in G}R_g$ be a $G$-graded $PI$-ring without unity which is a Jacobson radical (i.e $J(R)=R$). If $R_e$ is  nil-good, then $R$ is graded nil-good.
\end{thm}
\begin{proof}
According to \cite[Proposition 2.5]{1}, $J(R_e)$ is nil. Now, by \cite[Theorem 3]{8} we have  that $J(R)$ is nil since $R$ is by assumption $PI$-ring. On the other hand, $R$ is by assumption a Jacobson radical ring. Hence, $R=J(R)$ is a nil ring. In particular, every homogeneous element is nilpotent. Therefore, $R$ is graded nil-good.

\end{proof}

\begin{thm}\label{t4.6}
Let $R= \bigoplus_{g\in G}R_g$ be a $G$-graded $PI$-ring which is graded-local, and let $G$ be a finite group such that the order of $G$ is a unit in $R$. Assume that $R_gR_{g^{-1}}=0$ for every $g\in G \backslash\{e\}$. Then, if $R_e$ is nil-good, $R$ is graded nil-good.
\end{thm}
\begin{proof}
According to \cite[Proposition 2.5]{1} we have that $J(R_e)$ is nil. Moreover, by assumption $R_e/J(R_e)$ is nil-good. Now, according to \cite[Theorem 4.4]{3} we have that $J^g(R)=J(R)$. On the other hand, \cite[Theorem 3]{8} implies that $J(R)$ is nil, and hence $J^g(R)$ is graded-nil. Let $H$ be the homogeneous part of $R/J^g(R)$. According to the proof of Theorem 3.27 in \cite{15}, we have that $H\cong R_e/J(R_e)$. This means that every homogeneous element of $R/J^g(R)$ is graded nil-good. Hence, $R/J^g(R)$ is graded nil-good. Finally, according to Theorem \ref{t3.1}, $R$ is graded nil-good.

\end{proof}
\begin{thm}\label{4.7}

Let $R= \bigoplus_{g\in G}R_g$ be a $G$-graded ring of finite support where $G$ is a torsion free group, such that $R$ is a semiprimary ring with $R_e$ local nil-good. Then, $R$ is graded nil-good.
\end{thm}
\begin{proof}
Since $R_e$ is nil-good, by \cite[Proposition 2.8]{1} we have that $R_e/J(R_e)$ is nil-good too. By \cite[Proposition 9.6.4]{14} we have that $J^g(R)=J(R)$ and that $R/J(R)=R_e/J(R_e)$. Therefore, $R/J(R)$ is graded nil-good. On the other hand, $R$ is a semiprimary ring. Hence, $J(R)$ is nil. According to Theorem \ref{t3.1}, $R$ is graded nil-good.

\end{proof}

\section{Extensions of graded nil-good rings}
\subsection{Group rings and trivial ring extensions}
~~~~~In this subsection we investigate graded nil-good property of graded group rings and trivial ring extensions. Also, we establish some sufficient conditions for a group ring to be graded nil-good.
\begin{thm}\label{t4.1}
Let $A$ be a $G$-graded ring,  $E$ be a $G$-graded $A$-module, and let 
$R:=A\propto E$ be the trivial ring extension of $A$ by $E$. Then, $A$ is graded nil-good if and only if so is $R$.
\end{thm}
\begin{proof}
$\Longleftarrow)$ Assume that $R$ is graded nil-good. By \cite[Theorem 3.1]{17} we have that $A\cong \frac{R}{0\propto E}$. Hence, $A$ is graded nil-good as a graded homomorphic image of $R$.\\$\Longrightarrow)$ Assume that $A$ is graded nil-good. By \cite[Theorem 3.2]{17} we have that $Nil(R)=Nil(A)\propto E$. And according to \cite[Theorem 3.7]{17} we have that $U(R)=U(A)\propto E$. Now, let $(a,e)$ be a homogeneous element of $R$. Since $A$ is graded nil-good, then either $a\in Nil(A)$ or $a=u+n$ where $u\in U(A)$ and $n\in Nil(A)$. If $a\in Nil(A)$ then $(a,e) \in Nil(R)$. In the second case we have that $(a,e)=(u,e)+(n,0)$. Since $(u,e)\in U(R)$ and $(n,0)\in Nil(R)$ then  $(a,e)$ is graded nil-good. Finally, $R$ is graded nil-good.

\end{proof}
\begin{thm}\label{t4.2}
Let $A$ be a ring,  $E$ be an $A$-module, and let $R:=A\propto E$ be the trivial ring extension of $A$ by $E$. Then, $R$ is nil-good if and only if so is $A$.

\end{thm}

\begin{proof}
The same as the proof of Theorem \ref{t4.1}.

\end{proof}
Now, we are able to give a class of rings which are graded nil-good but not nil-good. 

\begin{ex}\label{e4.1}
Let $A$ be a graded nil-good ring that is not nil-good (e.g. Example \ref{e3.2}). Let $E$ be any graded $A$-module, and let $R:=A \propto E$. According to Theorem \ref{t4.1} and Theorem \ref{t4.2} $R$ is a graded nil-good ring but not a nil-good ring.
\end{ex}

Next we deal with the graded nil-good property of graded group rings. We recall that if $G$ is a group and $H$ is a normal subgroup of $G$ then a
$G$-graded ring $R= \bigoplus_{g\in G} R_g$ can be viewed as $G/H$-graded ring $R= \bigoplus_{C\in G/H} R_C$ where $R_C= \bigoplus_{x\in C} R_x$ (see for instance \cite{9,14}).

\begin{thm}\label{t4.3}

Let $R$ be a $G$-graded  ring where $G$ is a locally finite $p$-group, and let $H$ be a normal subgroup of $G$. Assume that $p$ is nilpotent in $R$. If $R$ is graded nil-good as a $G/H$-graded ring then $R[H]$ is  graded nil-good as a $G/H$-graded ring.
\end{thm}

\begin{proof}
Following the proof of Theorem 2.3 in \cite{10}, we may assume that $H$ is a finite $p$-group. According to \cite[page 180]{14}, the augmentation mapping $R[H] \longrightarrow R$ given by $ \sum_{h\in H}r_hh \mapsto \sum_{h\in H}r_h$, where $R[H]$ is considered as a $G/H$-graded ring, is a surjective degree-preserving morphism. Therefore, the kernel of the augmentation mapping, that is the augmentation ideal $\Delta(R[H])$, is homogeneous. This means that $R[H]/\Delta(R[H])$ is a $G/H$-graded ring. Moreover, $R[H]/\Delta(R[H])$ and $R$ are isomorphic as a $G/H$-graded rings. Hence, $R[H]/\Delta(R[H])$ is graded nil-good. Now, since $p$ is nilpotent we have that $\Delta(R[H])$ is nilpotent by \cite[Theorem 9]{11}, and therefore graded-nil. Applying Theorem \ref{t3.1} completes the proof.

\end{proof}
\begin{cor}\label{c4.1}
Let $R$ be a $G$-graded ring where $G$ is a locally finite $2$-group, and let $H$ be a normal subgroup of $G$. Assume that $R_e$ is a nil-clean ring. If $R$ is graded nil-good as a $G/H$-graded ring then $R[H]$ is  graded nil-good as a $G/H$-graded ring.

\end{cor}
\begin{proof}
Since $R_e$ is by assumption nil-clean, we have that $2$ is nilpotent by \cite[Proposition 3.14]{12}. Now, the previous theorem completes the proof.

\end{proof}
\begin{prop}\label{c4.2}
Let $R$ be a $G$-graded  ring, where $G$ is a locally finite $p$-group such that $p$ is nilpotent in $R$. If $R$ is graded nil-good then $R[G]$ is  graded nil-good.
\end{prop}

\begin{proof}
Following the proof of \cite[Theorem 2.]{10}, we can assume that $G$ is a finite $p$-group. On the other hand, by assumption $p$ is nilpotent in $R$, then by \cite[Theorem 9]{11} the augmentation ideal $\Delta(R[G])$ is nilpotent. Moreover, $R[G]/\Delta(R[G])$ and $R$ are isomoprhic as rings, hence by Theorem \ref{t3.1} $R[G]$ is graded nil-good.
\end{proof}

\begin{thm}\label{t4.4}
Let $R=\bigoplus_{g\in G}R_g$ be a $G$-graded ring such that the units and nilpotents of $R$ are all homogeneous. If $R[G]$ is graded nil-good then $R$ is graded nil-good.
\end{thm}
\begin{proof}
Since $R[G]$ is graded nil-good, $(R[G])_e$ is nil-good. According to \cite[Proposition 2.1 (4)]{13}, the mapping $f:R\longrightarrow (R[G])_e $, given by $f( \sum_{g\in G}r_g)=\sum_{g\in G}r_gg^{-1}$ is a ring isomorphism. Therefore, $R$ is nil-good. On the other hand, all the units and nilpotents of $R$ are homogeneous, hence $R$ is graded nil-good. 

\end{proof}
\begin{cor}\label{ex3.2}
Let $R$ be a $G$-graded ring where $G$ is a locally finite $p$-group such that $p$ is nilpotent in $R$, and the units and nilpotents of $R$ are homogeneous. If $(R[G])_e$ is a nil-good ring then $R[G]$ is a graded nil-good ring.
.
\end{cor}
\begin{proof}
We have that $R \cong (R[G])_e$ so $R$ is also nil good. By assumption, $R$ is a graded nil-good ring. Hence, by Corollary \ref{c4.2}, $R[G]$  will be a graded nil-good ring. 
\end{proof}

\subsection{Matrix rings} 
 In \cite[Corollary 4.2]{16}, it has been shown that if $R$ is a $G$-graded 2-good ring, then $M_n(R)(\overline{\sigma})$ is graded 2-good for every natural number $n$ and for every $\overline{\sigma}\in G^n$. Also, in \cite[Theorem 3.18]{15} it is proved that a matrix ring over a graded clean ring is also graded clean.  The author in \cite[Corollary 4.3]{16} has found a similar result concerning  graded 2-nil-good rings which are cross product. 

In this subsection, we try to obtain a similar result for graded nil-good rings.

The example below shows that if $R$ is a $G$-graded nil-good ring then $M_n(R)(\overline{\sigma})$ is not necessarily graded nil-good for every natural number $n$ and for every $\overline{\sigma}\in G^n$.

\begin{ex}\label{e4.3}
Let $G=\{e,g\}$ be a group of order two and let $R:=\mathbb{Z}_2\propto \mathbb{Z}_2$. We have that $R=(\mathbb{Z}_2\propto 0 )\bigoplus (0 \propto \mathbb{Z}_2)$ is a $G$-graded ring with $R_e=\mathbb{Z}_2\propto 0$ and $R_g=0 \propto \mathbb{Z}_2$. Homogeneous elements of $R$ are $(0,0)$, $(1,0)$ and $(0,1)$. We have that $(1,0) \in U(R)$ and $(0,1)\in Nil(R)$, and so $R$ is graded nil-good.\\Let $\overline{\sigma}=(e,e)$. We claim that $M_2(R)(\overline{\sigma})$ is not graded nil-good. Indeed, it is easy to show that the homogeneous element $M=\begin{pmatrix}
(1,0) &(0,0) \\ (0,0) &(0,0)
\end{pmatrix} \in M_2(R)_e(\overline{\sigma})$ is not graded nil-good. In fact, if we suppose that $M$ is graded nil-good, we obtain that $(1,0)$ is nilpotent which is a contradiction since it is a unit of $R$.
\end{ex}

Now, we give a sufficient condition for the matrix ring to be graded nil-good. First, we recall the notion of matrix "\textit{in good form}" defined in \cite{22}.
\begin{defo}[\cite{22}]
Let $R$ be a ring and $n \geq 2$  and let $\begin{pmatrix}
A&\beta \\ \gamma &d
\end{pmatrix}\in M_n(R)$ where $A\in M_{n-1}(R)$ and $d\in R$.  We say that the matrix $\begin{pmatrix}
A&\beta \\ \gamma &d
\end{pmatrix}$ is "in good form" if $A$ is nonzero and $d$ is also nonzero.
\end{defo}
 Next, we give a graded version of the notion \textit{fine rings} introduced in \cite{22}.
\begin{defo}
Let $R$ be a $G$-graded ring. $R$ is said to be graded fine  if every nonzero homogeneous element of $R$ can be written as a sum of a homogeneous unit and a  homogeneous nilpotent.
\end{defo}

Before, presenting the main result of this section, we need the next two Lemmas concerning the graded jacobson radical of matrix rings.

\begin{lemma}\label{l4.1}
If $R_1,R_2,\dots,R_n$  be $n$ $G$-graded rings and $A:=R_1\times \dots \times R_n$. Then, $J^g(A)=J^g(R_1)\times \dots \times J^g(R_n)$.
\end{lemma}

\begin{proof}

This follows directly from the fact that  graded-maximal right ideals of $A$ are exactly the ideals of the form $M_1\times \dots \times M_n$, where $M_i$ is a graded-maximal right ideal of $R_i$ for each $i$.

\end{proof}

\begin{lemma}\label{l4.2}
Let $R$ be a $G$-graded ring and $\overline{\sigma}=(e,e,\dots,e)$. For any natural number $n$, we have that: $$J^g(M_n(R))= M_n(J^g(R))$$

\end{lemma}

\begin{proof}

Let $\phi:R^{n^2} \longrightarrow M_n(R)$ be the canonical isomoprhism. It is clean that $\phi$ is a degree-preserving isomorphism. Hence, $J^g(M_n(R))=J^g\big(\phi (R^{n^2})\big)=\phi \big(J^g(R^{n^2})\big)$. On the other hand, according to Lemma \ref{l4.1} we have that $J^g(R^{n^2})=(J^g(R))^{n^2}$. Therefore, $J^g(M_n(R))=\phi \big((J^g(R))^{n^2} \big)=M_n(J^g(R))$, and  the proof is completed.

\end{proof}

\begin{thm}\label{t4.8}

Let $R$ be a $G$-graded commutative ring of finite support. Assume  $1=u+v$ where $u,v \in U(R_e)$. If $R$ is graded nil-good, then $M_2(R)(\overline{\sigma})$ is graded nil-good where $\overline{\sigma}=(e,e)$.
\end{thm}
\begin{proof}
Since $R$ is commutative, by Proposition \ref{p3.2} (2) we have that every nonzero homogeneous element of $\overline{R}=R/J^g(R)$ is unit. Let $M$ be a homogeneous nonzero matrix of $M_2(\overline{R})(\overline{\sigma})$. According to \cite[Corollary 3.7]{22}, $M$ is similar to a matrix in good form. Hence, there exist a invertible matrix $V$  and a matrix $A$ in good form of $M_2(\overline{R})(\overline{\sigma})$  such that $M=VAV^{-1}$. Following the proof of \cite[Corollary 3.7]{22}, we deduce that $V$ is invertible in $M_2(\overline{R}_e)$ and $A$ is homogeneous. Now, since $A$ is in good form we have that $A=\begin{pmatrix}u_1&b \\c & u_2 \end{pmatrix}=\begin{pmatrix}u_1&b \\0& u_2 
\end{pmatrix}+\begin{pmatrix}0&0\\c&0
\end{pmatrix}=U+N$ where $u_1,u_2\in U(\overline{R})$ is a graded nil-good decomposition of $A$. Hence, $M=VUV^{-1}+VNV^{-1}$, since $VUV^{-1}$ is invertible and $VNV^{-1}$ is nilpotent and both are homogeneous, we deduce that $M$ is graded fine which implies that $M_2(\overline{R})(\overline{\sigma})$ is graded fine and so it is graded nil-good. Using the homomorphism $f: M_2(R)\longrightarrow M_2(\overline{R})$ defined by $f\begin{pmatrix}
a&b\\
c&d
\end{pmatrix}= \begin{pmatrix}
\bar a&\bar b\\
\bar c&\bar d
\end{pmatrix}$ we find that $M_2(\overline{R})(\overline{\sigma})\cong M_2(R)/M_2(J^g(R)$ and so $M_2(R)/M_2(J^g(R))$ is graded nil-good. On the other hand, by Lemma \ref{l4.2} $M_2(J^g(R)) = J^g(M_2(R))$. Moreover, by \cite[Theorem 2.19]{1} we have that $M_2(R)_e(\overline{\sigma})=M_2(R_e)$ is nil good since $R_e$ is nil-good, this implies that $J\big(M_2(R)_e(\overline{\sigma})\big)$ is nil (by \cite[Proposition 2.5]{1}). In addition, by Amitsur-Levitski theorem (see \cite{19}) we have  that $M_2(R)$ is $PI$-ring. Now according to \cite[Theorem 3]{8} we have that $J(M_2(R))$ is nil. Since $R$ has a finite support, then by \cite[Corollary 2.9.4]{14} we get that $J^g(M_2(R)) \subseteq J(M_2(R))$. Hence, $J^g(M_2(R))$ is graded-nil. Finally, by Theorem \ref{t3.1} $M_2(R)(\overline{\sigma})$ is graded nil-good.

\end{proof}

\begin{cor}\label{c4.3}
Let $R$ be a $G$-graded commutative ring of finite support. Assume that $1=u+v$ where $ u,v \in U(R_e)$. If $R$ is graded nil-good and $n$ a natural number, then $M_n(R)(\overline{\sigma})$ is graded nil-good where $\overline{\sigma}=(e,\dots,e)\in G^n$.
\end{cor}
\begin{proof}
Using mathematical induction on $n$, we will prove that $M_n(\overline{R})(\overline{\sigma})$ is graded fine where $\overline{R}=R/J^g(R)$.\\For $n=2$, see the proof of Theorem \ref{t4.8}.\\Now, assume that $n\geq 2$ and that the claim holds for $M_n(\overline{R})(\overline{\sigma}) $. According to \cite[Proposition 3.9]{22}, every nonzero homogeneous matrix of $M_{n+1}(\overline{R})(\overline{\sigma})$ is similar to a homogeneous matrix in good form. Moreover, the change of basis matrix has to be from $M_{n+1}(\overline{R}_e)$.\\
Let $A\in M_{n+1}(\overline{R})(\overline{\sigma}) $ be a homogeneous matrix in good form. We have that $A=\begin{pmatrix}M&\beta \\ \gamma & d  \end{pmatrix}$ where $M\in M_n(\overline{R})$ nonzero and $d\in U(\overline{R})$. Since by assumption $M_n(\overline{R})(\overline{\sigma})$ is graded fine, then $M=U+N$ where $U$ (resp. $N$) is an invertible (resp. a nilpotent) matrix of $M_n(\overline{R})$. Hence, $A=\begin{pmatrix}U&\beta \\ 0 & d \end{pmatrix}+\begin{pmatrix} N&0\\ \gamma &0 \end{pmatrix}$ is a graded fine decomposition of $A$. Thus, we deduce that $M_n(\overline{R})(\overline{\sigma})$ is graded fine. Now, following the second part of the proof of Theorem \ref{t4.8}, we get that that $M_n(R)(\overline{\sigma})$ is graded nil-good.

\end{proof}

\end{document}